\documentclass[12pt]{amsart}
\usepackage[utf8]{inputenc}
\usepackage[english]{babel}
\usepackage[T1]{fontenc}
\usepackage{amsmath, amssymb, amsthm}

\usepackage{tikz}
\usepackage{tikz-cd}
\usetikzlibrary{decorations.pathmorphing}

\usepackage{environ}

\def\temp{&} \catcode`&=\active \let&=\temp

\NewEnviron{mycd}[1][]{%
  \begin{tikzpicture}[baseline=(current bounding box.west)]
    \node[inner sep=0, outer sep=0] {\begin{tikzcd}[#1]
      \BODY
    \end{tikzcd}};
  \end{tikzpicture}}
\usepackage{graphicx}

\usepackage{clock}
\usepackage{hyperref}

\newtheorem{lemma}{Lemma}[section]
\newtheorem{prop}[lemma]{Proposition}
\newtheorem{theorem}[lemma]{Theorem}
\newtheorem{claim}[lemma]{Claim}
\theoremstyle{definition}
\newtheorem{defi}[lemma]{Definition}
\theoremstyle{remark}
\newtheorem{rem}[lemma]{Remark}

\DeclareMathOperator{\can}{can}
\DeclareMathOperator{\im}{im}

\newcommand{\map}[3]{#1\colon #2 \to #3}
\newcommand{\setof}[2]{\{#1\colon #2\}}
\newcommand{\norm}[1]{\|#1\|}
\newcommand{\eps}{\varepsilon}

\title{A separable Fr\'{e}chet space of almost universal disposition}

\author{C.~Bargetz}
\address{The Technion---Israel Institute of Technology, Haifa, Israel (present address) and University of Innsbruck, Austria}

\author{J.~K\k{a}kol}
\address{Adam Mickiewicz University, Pozna\'n, Poland and Institute of Mathematics, Czech Academy of Sciences, Czech Republic}
\thanks{The second named author was supported by Generalitat Valenciana, Conselleria d'Educació, Cultura i Esport, Spain, Grant PROMETEO/2015/058 and by the GA\v{C}R project 16-34860L and RVO: 67985840.}

\author{W.~Kubi\'{s}}
\address{Institute of Mathematics, Czech Academy of Sciences, Czech Republic and Cardinal Stefan Wyszy\'nski
University, Warsaw, Poland}
\thanks{The third author was supported by the GA\v{C}R project 16-34860L and RVO: 67985840.}

\begin{document}

\begin{abstract}
  The Gurari\u{\i} space is the unique separable Banach space $\mathbb{G}$ which is of almost universal 
  disposition for finite-dimensional Banach spaces, which means that for every $\varepsilon>0$, for all  
  finite-dimensional normed spaces $E \subseteq F$, for every isometric embedding $\map{e}{E}{\mathbb{G}}$ 
  there exists an $\varepsilon$-isometric embedding $\map{f}{F}{\mathbb{G}}$ such that $f \restriction E = e$.
  We show that $\mathbb{G}^{\mathbb{N}}$ with a special sequence of semi-norms is of almost  
  universal disposition for finite-dimensional graded Fr\'echet spaces. The construction relies heavily 
  on the universal operator on the Gurari\u{\i} space, recently constructed by
  Gar\-bu\-li\'n\-ska-W\k egrzyn and the third author.
  This yields in particular that $\mathbb{G}^{\mathbb{N}}$ is universal in the class of all separable Fr\'echet spaces.

\end{abstract}

\maketitle

\tableofcontents

\section{Introduction}
A remarkable result of Banach and Mazur \cite{Rol1972MetricLinearSpaces} states that the separable Banach 
space $\mathcal{C}[0,1]$ is universal for separable Banach spaces. The above theorem has been extended by 
Mazur and Orlicz to the class of separable Fr\'echet spaces, i.e. metrizable and complete locally convex 
spaces: They proved that the separable Fr\'{e}chet space $\mathcal{C}(\mathbb{R})$ is universal in the class 
of all separable Fr\'{e}chet spaces, see again \cite{Rol1972MetricLinearSpaces}. An essential progress of 
the research around the Banach-Mazur theorem is due to  Gurari\u{\i}.
The Gurari\u{\i} space constructed by Gurari\u{\i} \cite{Gur1966UniversalPlacement} in 1965, is the separable Banach space
$\mathbb{G}$ of ``almost universal disposition for finite-dimensional spaces'' that is:
\begin{enumerate}
\item[(G)] For every $\varepsilon>0$, for every finite-dimensional normed spaces $E \subseteq F$,
for every isometric embedding $\map{e}{E}{\mathbb{G}}$ there exists an $\varepsilon$-isometric
embedding $\map{f}{F}{\mathbb{G}}$ such that $f \restriction E = e$.
\end{enumerate}
Moreover, if $Y$ is any other separable Banach space fitting this definition, then there exists
a linear isomorphism $u\colon\mathbb{G} \to Y$ with $\|u\|\cdot\|u^{-1}\|$ arbitrarily close to 1.
Lusky \cite{Lus1976GurarijUnique} proved that all separable Banach spaces of almost universal disposition are isometric,
see also \cite{KS13:Gurarii} for a simpler proof. Recall, that a linear operator $f\colon E\to F$ between Banach
spaces $E$ and $F$ is an \emph{$\varepsilon$-isometric embedding} for $\varepsilon > 0$ if
\[
(1+\varepsilon)^{-1}\cdot\|x\|\leq \|f(x)\| \leq (1+\varepsilon) \cdot \|x\|, \;  x\in E \setminus\{0\}
\]
Recall also that Gurari\u{\i} has already observed in \cite{Gur1966UniversalPlacement} that no separable
Banach space $E$ is of universal disposition, i.e. satisfies condition (G) with removing $\varepsilon$.

Being motivated by recent developments in the theory of Fr\'echet spaces we will study the concrete separable Fr\'echet space $\mathbb{G}^{\mathbb{N}}$
endowed with the product topology and
generated by two natural sequences of semi-norms, where $\mathbb{G}$ is the Gurari\u{\i} space. We
prove that $\mathbb{G}^{\mathbb{N}}$ is universal in the class of all separable Fr\'{e}chet spaces
although (as we show) that there is no separable Fr\'{e}chet space which is of universal disposition
for finite-dimensional Fr\'{e}chet spaces. Our main result states however that $\mathbb{G}^\mathbb{N}$ is the
unique separable Fr\'{e}chet space which is of almost universal
disposition for finite-dimensional Fr\'{e}chet spaces.

\section{Preliminaries}
Recall that a \emph{Fr\'{e}chet space} is a metrizable locally convex linear topological space which
is complete with respect to its canonical uniformity.
It is well-known that a complete separated locally convex topological vector space is a Fr\'{e}chet
space if and only if it satisfies the following condition:
There exists a sequence $\{\norm{\cdot}_{n}\}_{n\in\mathbb{N}}$ of semi-norms in $X$ such that
$B_i(v,r) = \setof{x \in X}{\max_{n\leq i}\norm{x-v}_n < r}$, where $i\in\mathbb{N}$, $v\in X$, $r>0$, generate
the topology of $X$. Fr\'echet spaces  have played an important role in functional analysis for a very long time.  Many vector
spaces of holomorphic, differentiable or continuous functions which appear in analysis and its applications carry a natural Fr\'echet
 topology. We refer the reader to an excellent survey of some recent developments in the theory of Fr\'echet spaces  and of their duals, see \cite{BB2003ModernFrechet}.
We refer also to \cite{Rol1972MetricLinearSpaces} for other fundamental facts and concepts
related with Fr\'{e}chet spaces.

In this paper Fr\'{e}chet spaces $E$  are considered with a fixed sequence of semi-norms. In the case of an increasing sequence,
i.e. if $\norm{\cdot}_1\leq\norm{\cdot}_2\leq \ldots$, we call (following Vogt \cite{Vog:OpFrechet}) $E$ endowed with this sequence
a \emph{graded Fr\'{e}chet space}.

Graded Fr\'{e}chet spaces have been studied in the context of the inverse function theorem of Nash and Moser, see~\cite{Ham1982NashMoserTheorem}, and in the context of tame Fr\'{e}chet spaces, see e.g.~\cite{Vog:OpFrechet, Vog1987TameSpaces, PV1995TameSplitting}. Note however that the concept of the category of graded Fr\'{e}chet spaces considered in~\cite{DV1998SplittingSmooth} differs from the notion of graded Fr\'{e}chet spaces used in this article. For a recent application of graded Fr\'{e}chet spaces, we refer the intersted reader to~\cite{Dav2013ApproxmationRegularity}.

In this section we construct an increasing sequence of semi-norm on $\mathbb{G}^{\mathbb{N}}$ 
under which $\mathbb{G}^{\mathbb{N}}$ is a graded Fr\'{e}chet space of almost universal disposition for 
finite-dimensional Fr\'{e}chet spaces. Our construction uses the following result being a special case of Theorem~6.5
in \cite{CGK2014:QuasiBanach}; property (2) below appears also as condition ($\ddagger$) on page 765 in
\cite{CGK2014:QuasiBanach}. It turns out that condition (2) determines $\pi$ uniquely up to linear isometries, although
this fact was not proved in \cite{CGK2014:QuasiBanach}.

\begin{theorem}\label{Thm:Ibuiwdsdoh}
  There exists a non-expansive linear operator $\pi\colon\mathbb{G}\to\mathbb{G}$ with the following
  properties.
  \begin{enumerate}
  \item[$(1)$] For every non-expansive operator $\map{T}{X}{\mathbb{G}}$ with $X$ a separable Banach
    space, there exists an isometric embedding $\map{i}{X}{\mathbb{G}}$ such that $T = \pi \circ i$.
  \item[$(2)$] For every $\eps>0$, for every finite-dimensional normed spaces $E \subseteq F$, for
    every non-expansive linear operator $\map T F \mathbb{G}$, for every isometric embedding
    $\map{e}{E}{\mathbb{G}}$ such that $T \restriction E = \pi \circ e$, there exists an
    $\varepsilon$-isometric embedding $\map{f}{F}{\mathbb{G}}$ such
    that
    \[
    \norm{f \restriction E - e} \leq \eps \quad \text{ and } \quad \norm{\pi \circ f - T} \leq \eps.
    \]
  \end{enumerate}
  Furthermore, $\ker \pi$ is linearly isometric to $\mathbb{G}$.
\end{theorem}

The last proposition may be applied to obtain the following useful

\begin{prop}\label{prop:ProjPi}
  The operator $\pi$ from Theorem~\ref{Thm:Ibuiwdsdoh} is a projection and it satisfies the following
  condition:
  \begin{enumerate}
  \item[$(2')$] For every $\eps>0$, for every finite-dimensional normed spaces $E \subseteq F$, for every
    non-expansive linear operator $\map T E \mathbb{G}$, for every isometric embedding $\map{e}{E}{\mathbb{G}}$ such
    that $T \restriction E = \pi \circ e$, there exists an $\varepsilon$-isometric embedding $\map{f}{F}{\mathbb{G}}$ such that
    \[
    f \restriction E = e \quad \text{ and } \quad \pi \circ f = T.
    \]
  \end{enumerate}
\end{prop}

\begin{proof}
  Taking $T = {\rm id}_{\mathbb{G}}$ in Theorem~\ref{Thm:Ibuiwdsdoh}, we see that $\pi$ is a projection.
  From now on, we shall identify the range of $\pi$ with a suitable subspace of its domain.
  In order to show ($2'$), we need to use the following fact which is easy to prove (see the proof of
  \cite[Theorem 2.7]{GK11:Gurarii} for more details).

  \begin{claim}
    Let $F$ be a finite-dimensional normed space and let $S = \{ v_0, \dots, v_n \}$ be a linear basis of $F$.
    Then for every $\varepsilon>0$ there exists $\delta>0$ such that for every pair of linear operators
    $\map{f,g}{F}{X}$ into an arbitrary Banach space $X$, the following implication holds:
    \[
    \max_{v \in S} \norm{f(v)-g(v)} \leq \delta \implies \norm{f-g} \leq \varepsilon.
    \]
  \end{claim}

  Now assume that $E \subseteq F$ and $S \cap E$ is a linear basis of $E$.
  Fix $\varepsilon>0$ and let $\delta>0$ be as mentioned in the claim. We
  apply property (2) from Theorem~\ref{Thm:Ibuiwdsdoh} with $\delta$ instead of $\varepsilon$.
  This provides a map  $\map f F \mathbb{G}$.
  Define $f'$ so that $f' \restriction E = e$ and $f'(v) = f(v)$ for every $v \in S \setminus E$.
  These conditions specify $f'$ uniquely and by the claim we have that $f'$ is a $2 \varepsilon$-isometric
  embedding.
  Furthermore, we have that $\norm{\pi \circ f' - T} \leq 2\varepsilon$.
  We have proved that for every $r>0$ there is an $r$-isometric embedding $\map{f}{F}{\mathbb{G}}$ extending
  $e$ and such that $\norm{\pi \circ f - T} \leq r$.
  Let us use this fact for $r = \delta$, where $\delta$ and $\varepsilon$ are as before.
  We obtain a $\delta$-isometric embedding $\map{f}{F}{\mathbb{G}}$ extending $e$ and satisfying
  $\norm{\pi \circ f - T} \leq \delta.$
  Fix $v \in S \setminus E$.
  Then the vector $w_v = \pi(f(v)) - T(v)$ has norm $\leq \delta$.
  Define $\map{f'}{F}{\mathbb{G}}$ so that $f' \restriction E = f$ and
  \[
  f'(v) = f(v) - w_v, \qquad v \in S \setminus E.
  \]
  Note that $\norm{f'(v) - f(v)} \leq \delta$ for $v \in S$, therefore $\norm{f' - f} \leq \varepsilon$.
  Finally, $\pi f'(v) = \pi f(v) - \pi f(v) + \pi T(v) = T(v)$.
  It follows that $\pi f' = T$.
\end{proof}

Now we are ready to construct a graded sequence of semi-norms on $\mathbb{G}^{\mathbb{N}}$. Having in mind  
the last two results mentioned above, we conclude that there  is a norm $\norm{\cdot}_2$ on the product
$\mathbb{G}\times\mathbb{G} \cong (\im \pi) \times (\ker \pi)$ satisfying
\[
\|x\|_{\mathbb{G}} = \|(x,y)\|_1 \leq \|(x,y)\|_2
\]
since $x=\pi(x,y)$ and $\pi$ is non-expansive. In addition, $(\mathbb{G}^{2}, \|\cdot\|_{2})$ is
isometric to $\mathbb{G}$. We identify $\mathbb{G}\times\mathbb{G}$ with $\mathbb{G}$ and use the
notation $\pi_2$ for $\pi$ in this case to stress that we consider it as mapping from $\mathbb{G}^2$
to $\mathbb{G}$.
Inductively, for all $n\in\mathbb{N}$, we get a norm $\|\cdot\|_n$ on $\mathbb{G}^{n}$ satisfying 
$\|x\|_1\leq \cdots \leq \|x\|_{n-1}\leq \|x\|_n$ for all $x\in\mathbb{G}^{n}$ and 
$(\mathbb{G}^{n},\|\cdot\|_{n})$ is isometric to $\mathbb{G}$.
Therefore this construction provides an increasing sequence of semi-norms on $\mathbb{G}^{\mathbb{N}}$ 
as claimed. We use the notation $\pi_n\colon \mathbb{G}^{n}\to\mathbb{G}^{n-1}$ for the universal operator 
$\pi$ if we want to stress that we consider it as projection from $\mathbb{G}^{n}$ to $\mathbb{G}^{n-1}$,
i.e., a projection onto the first $n-1$ components.

In order to formulate a condition similar to~(G) for Fr\'{e}chet spaces we need to define the corresponding 
concept of $\varepsilon$-isometries for Fr\'{e}chet spaces.

\begin{defi}\label{def:epsIsom}
  Let $(X,\{\|\cdot\|_i\}_{i\in\mathbb{N}})$ and $(Y,\{\|\cdot\|_i\}_{i\in\mathbb{N}})$ be Fr\'{e}chet spaces
  with fixed sequences of semi-norms. A mapping $f\colon X\to Y$ is called an \emph{$\varepsilon$-isometric
  embedding} iff it is an embedding and
  \begin{equation}
    (1+\varepsilon)^{-1}\|x\|_i \leq \|f(x)\|_i \leq (1+\varepsilon)\|x\|_i
  \end{equation}
  holds for all $i\in\mathbb{N}$ an all $x\in X$. The mapping $f$ is called an \emph{isometric embedding} iff
  \begin{equation}
    \|f(x)\|_i  = \|x\|_i
  \end{equation}
  holds for all $i\in\mathbb{N}$ an all $x\in X$.
\end{defi}
Now we are ready to formulate the
analogue of condition~(G) for Fr\'{e}chet spaces.

\begin{defi}
  A Fr\'{e}chet space $E$ is \emph{of almost universal disposition for finite-dimensional Fr\'{e}chet
  spaces} if for all $\varepsilon>0$ and for all finite-dimensional Fr\'{e}chet spaces $X\subseteq Y$
  with an isometric embedding $f_0\colon X\to E$ there exists an $\varepsilon$-isometric embedding
  $f\colon Y\to E$ satisfying $f\restriction E = f_0$.
\end{defi}

\section{A graded Fr\'echet space of almost universal disposition}

We show that the space $\mathbb{G}^{\mathbb{N}}$ equipped with the graded sequence $\{\norm{\cdot}_n\}_{n\in\mathbb{N}}$ 
of semi-norms defined above (coming from the universal operator $\pi$) is of almost universal disposition for finite-dimensional graded Fr\'{e}chet spaces. We need the following

\begin{lemma}\label{lem:embed}
  Let $(X,\{\|\cdot\|_i\}_{i\in\mathbb{N}}$ and $(Y,\{\|\cdot\|_i\}_{i\in\mathbb{N}})$ be Fr\'{e}chet spaces
  with fixed semi-norms and $\iota\colon X\to Y$ an isometric embedding. Then for all $i\in\mathbb{N}$
  the mapping
  \begin{displaymath}
    \iota_i\colon X/\ker\|\cdot\|_i \to Y/\ker\|\cdot\|_i, \bar{x}\mapsto \overline{\iota(x)}
  \end{displaymath}
  is a well-defined isometric embedding. Moreover, the diagram
  \begin{center}		
    \begin{mycd}
      \prod X/\ker\|\cdot\|_i \arrow[hookrightarrow]{r}{\prod\iota_i} & \prod Y/\ker\|\cdot\|_i\\
      X \arrow[hookrightarrow]{r}{\iota} \arrow[hookrightarrow]{u} & \arrow[hookrightarrow]{u} Y
    \end{mycd}
  \end{center}
  is commutative.
\end{lemma}

\begin{proof}
  As $\iota$ is an isometric embedding, we have $\|\iota(x)\|_i = \|x\|_i$. Hence
  $x\in\ker\|\cdot\|_i$ iff $\iota(x)\in\ker\|\cdot\|_i$, i.e. $\iota_i$ is well-defined. By
  definition we have
  \begin{displaymath}
    \|\iota_i(\bar{x})\|_i = \left\|\overline{\iota(x)}\right\|_i = \|\iota(x)\|_i = \|x\|_i
    = \|\bar{x}\|_i
  \end{displaymath}
  and hence $\iota_i$ is an isometric embedding. The commutativity of the diagram directly follows
  from the definition of $\iota_i$.
\end{proof}

In addition, we also need the following technical
\begin{lemma}\label{lem:PushOut}
  Let $X\subseteq Y$ and $A$ be finite-dimensional Banach spaces, $Z$ a Banach space, $e\colon X\to A$
  an isometric embedding, $T\colon Y\to Z$ a linear operator with $\|T\|\leq r$, $r>1$, and $\pi\colon A\to Z$
  a non-expansive operator such that the diagram
  \[
  \begin{mycd}
    A \arrow{r}{\pi}  & Z \\
    X \arrow[hookrightarrow]{r} \arrow[hookrightarrow]{u}{e} & Y \arrow[swap]{u}{T}
  \end{mycd}
  \]
  commutes.
  There exists a finite-dimensional Banach space $C$, an isometric embedding $i_A\colon A\to C$,
  an $(r-1)$-isometric embedding $i_Y\colon Y \to C$ and a non-expansive operator $\pi'\colon C\to Z$
  such that the diagram
  $$
  \begin{mycd}
     \phantom{a}& & Z \\
     A \arrow{rru}{\pi} \arrow[swap, hookrightarrow]{r}{i_A} & C \arrow[swap]{ru}{\pi'} &  \\
     X \arrow[hookrightarrow]{rr} \arrow[hookrightarrow]{u}{e} & & Y \arrow[swap]{uu}{T} \arrow[rightsquigarrow]{lu}{i_Y}
   \end{mycd}
   $$
   is commutative.
\end{lemma}

\begin{proof}
  We define $C:= (A\oplus Y)/\Delta$ with $\Delta=\{(e(x),-x)\colon x\in X\}$ equipped with the
  norm
  \[
  \|(a,y)\| = \inf_{x\in X} \{\|a-e(x)\|_A+r\|x+y\|_Y\}.
  \]
  First we show that $i_A$ is an isometry. For $a\in A$ we obtain
  \begin{align*}
    \|i_A(a)\| &= \inf_{x\in X} \{\|a-e(x)\|_A+r\|x\|_Y\} \\
    &=  \inf_{x\in X} \{\|a-e(x)\|_A+\|e(x)\|_A+(r-1)\|x\|_Y\} \\
    & \geq \inf_{x\in X} \{\|a\|_A +(r-1)\|x\|_Y\} = \|a\|_A
  \end{align*}
  and, by setting $x=0$, $\|i_A(a)\| \leq \|a\|_A$. For $y\in Y$ we get
  \[
  \|i_Y(y)\| = \|(0,y)\| = \inf_{x\in X} \{\|-e(x)\|_A+r\|x+y\|_Y\} \leq r\|y\|_Y
  \]
  by setting $x=0$, and
  \begin{align*}
    \|i_Y(y)\| &= \inf_{x\in X} \{\|-e(x)\|_A+r\|x+y\|_Y\} \\
    &= \inf_{x \in X} \{\|-x\|_Y+\|y+x\|_Y+(r-1)\|x+y\|_Y\}\\
    & \geq \inf_{x\in X} \{\|y\|_Y+(r-1)\|x+y\|_Y\} \geq  \|y\|_Y \geq \frac{1}{r}\|y\|_Y,
  \end{align*}
  again using the triangle inequality.
  The linear operator $\pi'\colon A\to Z$ can be defined as
  \begin{align*}
    \pi'((a,y)) &= \pi(a-e(x))+T(x+y) = \pi(a) -\pi(e(x))+T(x)+T(y)\\ 
                & = \pi(a)+T(y).
  \end{align*}
  It satisfies
  \begin{align*}
  \|\pi'((a,y))\|_Z &= \|\pi'((a-e(x),x+y))\|_Z \\
                    & \leq \|\pi(a-e(x))\|_Z+\|T(x+y)\|_Z\\
                    & \leq \|a-e(x)\|_A+r\|x+y\|_Y
  \end{align*}
  for all $x\in X$ and hence $\|\pi'((a,y))\|\leq \|(a,y)\|$.
\end{proof}

\begin{theorem}
  The space $\mathbb{G}^{\mathbb{N}}$ equipped with the graded sequence $$\{\norm{\cdot}_n\}_{n\in\mathbb{N}}$$
  of semi-norms defined above is of almost universal disposition for finite-dimensional graded Fr\'{e}chet
  spaces.
\end{theorem}

\begin{proof}
  Given an isometric embedding $f\colon X\to\mathbb{G}^{\mathbb{N}}$ and $\varepsilon > 0$, we choose a 
  sequence $(\varepsilon_i)_{i\in\mathbb{N}}$ satisfying $\prod_{i=1}^{\infty}(1+\varepsilon_i) < 1 + \varepsilon$.
  We will use the notation $X_i := X/\ker \|\cdot\|_i$, $Y_i := Y/\ker\|\cdot\|_i$.
  By $f_i\colon X_i\to \mathbb{G}^{i}$ we denote the mapping induced by the isometric embedding 
  $f\colon X\to\mathbb{G}^{\mathbb{N}}$.

  As a first step, we can use that $X_1\subseteq Y_1$ are finite-dimensional Banach spaces and
  $f_1\colon X_1\to\mathbb{G}$ is an isometric embedding to obtain an $\varepsilon_1$-isometric
  embedding $g_1\colon Y_1\to\mathbb{G}$ which extends $f_1$.

  Now assume that we already have constructed a $\delta$-isometric embedding
  \[
  g_{n-1}\colon Y_{n-1} \to \mathbb{G}^{n-1},
  \]
  where $1+\delta = \prod_{i=1}^{n-1}(1+\varepsilon_i)$, extending $f_{n-1}\colon X_{n-1}\to\mathbb{G}^{n-1}$.

  We consider the diagram
  \begin{center}
    \begin{mycd}[column sep=2cm]
      f_n(X_n) \arrow[two heads]{r}{\pi\restriction {f_n(X_n)}}& \mathbb{G}^{n-1} \\
      X_n \arrow[hookrightarrow]{r} \arrow[hookrightarrow]{u}{f_n} & Y_n \arrow[swap]{u}{g_{n-1}\circ p_{n-1}^{n}}
    \end{mycd},
  \end{center}
  where $p_{n-1}^{n}\colon Y_n\to Y_{n-1}$ denotes the canonical projection. Note that it is commutative
  since $\pi\circ f_{n} =f_{n-1}$. From Lemma~\ref{lem:PushOut} we deduce the existence of a finite-dimensional Banach space $C$ such that the diagram
  \begin{center}
  \begin{mycd}
     \mathbb{G}^{n} \arrow{rr}{\pi}& & \mathbb{G}^{n-1} \\
     f_n(X_n) \arrow{rru}{\pi\restriction {f_n(X_n)}} \arrow[swap, hookrightarrow]{r}{i_{f_n(X_n)}}
     \arrow[hookrightarrow]{u}& C \arrow[swap]{ru}{\pi'} &  \\
     X_n \arrow[hookrightarrow]{rr} \arrow[hookrightarrow]{u}{f_n} & & Y_n \arrow[swap]{uu}{g_{n-1}\circ\pi_{n-1}^{n}}
     \arrow[rightsquigarrow]{lu}{i_{Y_n}}
   \end{mycd}
   \end{center}
   commutes and $\pi'$ is non-expansive. From Proposition~\ref{prop:ProjPi}, we may now conclude that
   there is an $\varepsilon_n$-isometric emebedding $\tilde{g}_n\colon C\to \mathbb{G}^{n}$ which extends
   both $f_n$ and $\pi'$. Hence $g_n = \tilde{g}_n\circ \iota_{Y_n}$ is an $\delta'$-isometric embedding,
   where $1+\delta'= \prod_{i=1}^{n}(1+\varepsilon_i)$, which extends both $f_n$ and $g_{n-1}$.

   Therefore by induction we get an $\varepsilon$-isometric embedding $g\colon Y\to \mathbb{G}^{\mathbb{N}}$ extending
   the embedding $f\colon X\to\mathbb{G}^{\mathbb{N}}$.
\end{proof}

\section{Uniqueness and Universality}

The aim of this section is to prove universality and the following uniqueness result for the space 
$\mathbb{G}^{\mathbb{N}}$. 

\begin{prop}\label{prop:ApprIsomUnordered}
  Let $E$ and $F$ be graded Fr\'{e}chet spaces which are of almost universal disposition for finite-dimensional
  graded Fr\'{e}chet spaces, $\varepsilon>0$, $X\subseteq E$ a finite-dimensional subspace and $f\colon X\to F$
  an $\varepsilon$-isometric embedding. Then there exists a bijective isometry $h\colon E\to F$ satisfying
  $\|h(x)-f(x)\|_{i}<4\varepsilon\|x\|_i$ for all $i\in\mathbb{N}$.
\end{prop}
First we need to show some additional lemmata used for the proof.
Let $(X,\{\|\cdot\|_{i}\}_{i\in\mathbb{N}})$ and $(Y,\{\|\cdot\|_{i}\}_{i\in\mathbb{N}})$ be Fr\'{e}chet spaces
with a fixed sequence of semi-norms and $f\colon X\to Y$ a linear mapping with the property
that for all $i\in\mathbb{N}$ there exists a constant $C_i>0$ such that $\|f(x)\|_i\leq C_i \|x\|_i$
holds for all $x\in X$. We can do this since we only want to consider mappings which are
isometries or at least $\varepsilon$-isometries.

We can now define
\[
\|f\|_i= \sup_{\|x\|_i=1} \|f(x)\|_i
\]
for all $i\in\mathbb{N}$.
  Note that the above  condition on $f$  is stronger than the continuity of $f$.

We need the following Lemma motivated by \cite[Lemma~2.2]{Gar13:Universal} for  Banach spaces.

\begin{lemma}\label{lem:corr}
  Let $(X,\{\|\cdot\|_{X,i}\}_{i\in \mathbb{N}})$ and $(Y,\{\|\cdot\|_{Y,i}\}_{i\in \mathbb{N}})$ be
  finite-dimensional graded Fr\'{e}chet spaces and let $\varepsilon>0$ and $f\colon X\to Y$ be
  an $\varepsilon$-isometric embedding. There exists a finite-dimensional graded Fr\'{e}chet space
  $Z$ and isometric embeddings $\iota\colon X\to Z$ and $j\colon Y\to Z$ such that
  $\|j\circ f-\iota\|_{i} \leq \varepsilon$ holds for all $i\in\mathbb{N}$.
\end{lemma}

\begin{proof}
  We set $Z = X \oplus Y$ equipped with the semi-norms
  \begin{multline*}
    \|(x,y)\|_{i}=\inf\{\|u\|_{X,i}+\|v\|_{Y,i}+\varepsilon\|w\|_{X,i}\colon x=u+w,\\
    y=v-f(w),u,w\in X, v\in Y\}.
  \end{multline*}
  We have
  \begin{displaymath}
    \|j(f(x))-\iota(x)\|_{i}=\|(x,-f(x))\|_{i}\leq \varepsilon \|x\|_{X,i}
  \end{displaymath}
  by taking $u=0$, $v=0$ and $w=x$. Hence $\|j\circ f-\iota\|_{i}\leq\varepsilon$ for all
  $i\in\mathbb{N}$. From
  \begin{displaymath}
    \frac{1}{1+\varepsilon}\geq 1-\varepsilon \Leftrightarrow 1 \geq 1-\varepsilon^2
  \end{displaymath}
  we may deduce $\|f(x)\|_{Y,i}\geq (1+\varepsilon)^{-1}\|x\|_{X,i}\geq (1-\varepsilon)\|x\|_{X,i}$ which
  we will need in the following.
  Now we show that $\iota$ is an isometric embedding. For $x\in X$ we have
  \begin{displaymath}
    \|\iota(x)\|_i = \|(x,0)\|_{i} \leq \|x\|_{X,i}.
  \end{displaymath}
  Setting $x=u+w$ and $0=v-f(w)$, we obtain
  \begin{align*}
    \|u\|_{X,i}+\|v\|_{Y,i}+\varepsilon\|w\|_{X,i} & = \|u\|_{X,i}+\|f(w)\|_{Y,i}+\varepsilon \|w\|_{X,i}\\
    &\geq \|u\|_{X,i}+(1-\varepsilon)\|w\|_{X,i}+\varepsilon\|w\|_{X,i}\\
    &= \|u\|_{X,i}+\|w\|_{X,i}\geq\|u+w\|_{X,i}=\|x\|_{X,i}
  \end{align*}
  and hence $\|\iota(x)\|_{i}=\|(x,0)\|_{i}\geq \|x\|_{X,i}$. Therefore $\iota$ is an isometric embedding.
  Analogously to above, we have $\|j(y)\|_i=\|(0,y)\|_i\leq \|y\|_{Y,i}$. Setting $0=u+w$ and $y=v-f(w)$,
  we obtain
  \begin{align*}
    \|u\|_{X,i}+\|v\|_{Y,i}+\varepsilon \|w\|_{X,i} &=\|w\|_{X,i}+\|v\|_{Y,i}+\varepsilon\|w\|_{X,i}\\
                                                    &=(1+\varepsilon)\|w\|_{X,i}+\|v\|_{Y,i}\\ 
                                                    & \geq \|f(w)\|_{Y,i}+\|v\|_{Y,i}\\
                                                    &\geq \|f(w)-v\|_{Y,i} = \|y\|_{Y,i}
  \end{align*}
  and hence $\|j(y)\|_i=\|(0,y)\|_i\geq \|y\|_{Y,i}$, i.e. $j$ is also an isometric embedding.
\end{proof}
We need also the next
\begin{lemma}\label{lem:bafarg}
  Let $E$ be a graded Fr\'{e}chet space which is of almost universal disposition for finite-dimensional 
  graded Fr\'{e}chet spaces, $X\subseteq E$ a finite-dimensional subspace and $\varepsilon>0$. Given an 
  $\varepsilon$-isometry $f\colon X\to Y$, for all $\delta>0$ there exists a $\delta$-isometry $g\colon 
  Y\to E$ with the property $\|g\circ f - \operatorname{id}_X\|_{X,i}<2\varepsilon$ for all $i\in\mathbb{N}$.
\end{lemma}

\begin{proof}
  Choose $0<\delta'<\min\{\delta, 1\}$. By Lemma~\ref{lem:corr} there exists a finite-dimensional
  Fr\'{e}chet space $Z$ and isometric embeddings $\iota\colon X\to Z$ and $j\colon Y\to Z$ such
  that $\|j\circ f-\iota\|_i\leq \varepsilon$ for all $i\in\mathbb{N}$. As $E$ is of almost
  universal disposition for finite-dimensional Fr\'{e}chet spaces, there exists a $\delta'$-isometric
  embedding $h\colon Z\to E$ with the property $h\circ \iota\restriction X=\operatorname{id}_X$. Setting
  $g=h\circ j$, we get that $g$ is a $\delta$-isometric embedding as it is the composition of a
  $\delta$-isometric embedding with an isometric embedding. Additionally for $x\in X$ we obtain
  \begin{align*}
    \|g(f(x))-x\|_{i}&=\|h(j(f(x)))-h(\iota(x))\|_{i}\leq (1+\delta')\|j(f(x))-\iota(x)\|_{i}\\
    & \leq \varepsilon(1+\delta')\|x\|_{i} < 2\varepsilon \|x\|_{i}
  \end{align*}
  for all $i\in\mathbb{N}$.
\end{proof}

Now we can use these results to show that $\mathbb{G}^{\mathbb{N}}$ is, up to isometry, uniquely
determined by the property of being of almost universal disposition for finite-dimensional Fr\'echet spaces.

\begin{proof}[Proof of Proposition~\ref{prop:ApprIsomUnordered}]
  Let $\{\varepsilon_n\}_{n\in\mathbb{N}}$ be a fixed sequence of positive real numbers satisfying a
  decay condition which will be specified later in the proof.

  We pick $\varepsilon_0=\varepsilon$ and set $X_0=X$, $Y_0=Y$ and $f_0=f$. By assumption the
  mapping $f_0\colon X_0\to Y_0$ is an $\varepsilon_0$-isometric embedding.

  Now assume that $X_i$, $Y_i$ and $f_i$ have already been constructed for $i\leq n$ and that
  also the mappings $g_i$ have been constructed for $i<n$. Using Lemma~\ref{lem:bafarg} we
  obtain an $\varepsilon_{n+1}$-isometric embedding $g_n\colon Y_n\to X_{n+1}$ satisfying
  \begin{equation}\label{eq:gfsmall}
    \|g_n(f_n(x))-x\|_{i}\leq 2\varepsilon_{n}\|x\|_i
  \end{equation}
  for all $i\in \mathbb{N}$. Here the
  space $X_{n+1}$ is defined as an appropriately enlarged $g_{n}[Y_n]$ such that $Y_{n-1}\subseteq
  Y_{n}$ and $\bigcup_{n\in\mathbb{N}}X_n$ is dense in $E$. Again by using Lemma~\ref{lem:bafarg}
  we get an $\varepsilon_{n+1}$-isometric embedding $f_{n+1}\colon X_{n+1}\to Y_{n+1}$ where
  $Y_{n+1}$ is chosen analogously to $X_{n+1}$. This mapping satisfies
  \begin{equation}\label{eq:fgsmall}
    \|f_{n+1}(g_{n}(y))-y\|_i \leq 2\varepsilon_{n+1}\|y\|_i
  \end{equation}
  for all $i\in\mathbb{N}$.

  Now for fixed $n$ and $x\in X_n$ we get
  \[
  \|f_{n+1}(g_n(f_n(x)))-f_n(x)\|_i\leq 2\varepsilon_{n+1} \|f_n(x)\|_i
  \leq 2\varepsilon_{n+1}(1+\varepsilon_n)\|x\|_i
  \]
  and
  \begin{align*}
  \|f_{n+1}(g_{n}(f_n(x)))-f_{n+1}(x)\|_i &\leq (1+\varepsilon_{n+1}) \|g_n(f_n(x))-x\|_i\\
                                          &\leq 2\varepsilon_n(1+\varepsilon_{n+1})\|x\|_i.
  \end{align*}
  Using the triangle inequality, we obtain
  \[
  \|f_{n+1}(x)-f_{n}(x)\|_i\leq (\varepsilon_n+2\varepsilon_n\varepsilon_{n+1}+\varepsilon_{n+1})2\|x\|_i
  \]
  from the inequalities above. Now we assume  that
  \[
  \varepsilon_0 + 2\varepsilon_0\varepsilon_1 + \varepsilon_1 +
  \sum_{n=1}^{\infty}(\varepsilon_n+2\varepsilon_{n}\varepsilon_{n+1}+\varepsilon_{n+1})
  < 2\varepsilon
  \]
  which implies that $\{f(x_n)\}_{n\in\mathbb{N}}$ is a Cauchy sequence in $F$.

  For $x\in\bigcup_{n\in\mathbb{N}}X_n$ we define $h(x) = \lim_{m\geq n} f_n(x)$ where $m$ is chosen
  such that $x\in X_m$. Then $h$ is an $\varepsilon_n$-isometry for all $n\in\mathbb{N}$ and
  hence an isometry which can be uniquely extended to an isometry on $E$, which will be denoted
  by $h$ as well. From the inequalities above we deduce
  \[
  \|f(x)-h(x)\|_{i} \leq 2\sum_{n=0}^{\infty}(\varepsilon_n+2\varepsilon_n\varepsilon_{n+1}+\varepsilon_{n+1})
  < 4\varepsilon.
  \]
  In order to show that $h$ is bijective, we repeat the above procedure to show that $\{g_n(y)\}_{n\geq m}$
  is a Cauchy sequence for all $y\in Y_m$. Again we may obtain an isometry $g_\infty\colon F\to E$.
  From the conditions~\eqref{eq:gfsmall} and~\eqref{eq:fgsmall} we may conclude $g_\infty\circ h =
  \operatorname{id}_E$ and $h\circ g_\infty = \operatorname{id}_F$.
\end{proof}

We conclude this section by showing that every graded Fr\'{e}chet space can be embedded isometrically 
into $\mathbb{G}^{\mathbb{N}}$.

\begin{theorem}
  The space $\mathbb{G}^{\mathbb{N}}$ is universal for separable Fr\'{e}chet spaces.
\end{theorem}

\begin{proof}
  Let $X$ be a separable Fr\'{e}chet space and $\{\|\cdot\|_{i}\}_{i\in\mathbb{N}}$ its fixed sequence of
  semi-norms. For all $i\in\mathbb{N}$ we denote by $X_i$ the normed space $X/\ker\|\cdot\|_{i}$ equipped
  with the norm $\|\cdot\|_{i}$ and by $\widetilde{X_i}$ its completion.
  From the universality of $\mathbb{G}$ we may deduce the existence of an isometric embedding 
  $f_1\colon \widetilde{X}_1\to\mathbb{G}$. 
  Now assume that we have an isometric embedding $f_i\colon\widetilde{X}_i\to \mathbb{G}^{i}$. Note that 
  since $\|x\|_i\leq\|x\|_{i+1}$ for all $x\in X$, the composition $T$ of $f_i$ with the canonical mapping 
  $\can^{i+1}_{i}\colon X_{i+1}\to X_{i}$ is non-expansive. 
  As $\mathbb{G}^{i+1}$, equipped with the semi-norm $\|\cdot\|_{i+1}'$, is isometric to 
  $\mathbb{G}$, we can use property~(1) of the universal projection $\pi_{i+1}$ given in Theorem~\ref{Thm:Ibuiwdsdoh} 
  to find an isometric embedding $f_{i+1}\colon \widetilde{X}_{i+1}\to\mathbb{G}^{i+1}$ so that the diagram
  \begin{center}
    \begin{mycd}
      \mathbb{G}^{i+1} \arrow{r}{\pi_{i+1}} & \mathbb{G}^{i}\\
      X_{i+1} \arrow[hook]{u}{f_i+1} \arrow[swap]{r}{\can^{i+1}_{i}} \arrow[bend right]{ru}{T} & X_{i} \arrow[hook,swap]{u}{f_i}
    \end{mycd}
  \end{center}
  is commutative. Hence,
  \[
  f\colon X\to\mathbb{G}^{\mathbb{N}}, \quad x\mapsto ((f_i(\can_i(x)))_{i})_{i\in\mathbb{N}}
  \]
  is an isometric embedding.
\end{proof}

\section{Final remarks}

Note that by \cite[Proposition~V.5.4]{Rol1972MetricLinearSpaces} the space $\mathcal{C}(\mathbb{R})$ is 
universal for all separable Fr\'{e}chet spaces. 
The following shows that, like in the case of $\mathcal{C}([0,1])$ for Banach spaces, the space 
$\mathcal{C}(\mathbb{R})$ is not of almost universal disposition for finite-dimensional Fr\'{e}chet spaces.

\begin{prop}
  Let $X$ be a hemicompact space and $\{K_i\}_{i\in\mathbb{N}}$ be a sequence of compact sets satisfying
  $K_{i}\subseteq K_{i+1}$ and $\bigcup_{i\in\mathbb{N}}K_i = X$. The space $\mathcal{C}(X)$ equipped with the
  sequence of semi-norms $\|\cdot\|_i$ where $\|f\|_{i} = \sup_{x\in K_i} |f(x)|$ is not of almost
  universal disposition for finite-dimensional Fr\'{e}chet spaces.
\end{prop}

\begin{proof}
  Assume for a contradiction that $\mathcal{C}(X)$ is a of almost universal disposition for finite-dimensional Fr\'{e}chet spaces.
  By Proposition~\ref{prop:ApprIsomUnordered} we deduce that $\mathcal{C}(X)\cong \mathbb{G}^{\mathbb{N}}$
  holds isometrically. Hence for all $i\in\mathbb{N}$ we obtain $\mathcal{C}(X)/\ker\|\cdot\|_{i}\cong
  \mathbb{G}$ isometrically. Observe that $\mathcal{C}(X)/\ker\|\cdot\|_{i}=\mathcal{C}(K_i)$. This shows that  $\mathcal{C}(K)$-space is a Gurari\u{\i} space, a contradiction with \cite[Corollary~5.4]{ACC2011:BanachSpaceUD}.
\end{proof}

\begin{prop}
  There is no separable Fr\'{e}chet space which is of universal disposition for finite-dimensional
  Fr\'{e}chet spaces.
\end{prop}

\begin{proof}
  Assume, for a contradiction, that $F$ is a separable Fr\'{e}chet space of universal disposition for
  finite-dimensional Fr\'{e}chet spaces. Hence $F$ is also of almost universal disposition for
  finite-dimensional Fr\'{e}chet spaces, and hence by Proposition~\ref{prop:ApprIsomUnordered}
  isometrically isomorphic to $\mathbb{G}^{\mathbb{N}}$. Therefore it is sufficient to show that
  $\mathbb{G}^{\mathbb{N}}$ is not of universal disposition.

  Let $X\subseteq Y$ be finite-dimensional Banach spaces. Setting $\|\cdot\|_{X,i}:=\|\cdot\|_{X}$ and
  $\|\cdot\|_{Y,i}:=\|\cdot\|_{Y}$ for all $i\in\mathbb{N}$, we obtain two finite-dimensional
  Fr\'{e}chet spaces. Assume, for a contradiction, that $\mathbb{G}^{\mathbb{N}}$ is of universal
  disposition for finite-dimensional Fr\'{e}chet spaces. Given an isometric embedding
  $f\colon (X,\|\cdot\|_{X})\to\mathbb{G}$ the product mapping
  \[
  X\to \mathbb{G}^{\mathbb{N}}, x\mapsto \{f(x)\}_{i\in\mathbb{N}}
  \]
  is an isometric embedding of Fr\'{e}chet spaces. Hence there would exist an isometric extension
  $g\colon Y\to \mathbb{G}^{\mathbb{N}}$. Therefore also the mapping $Y\to\mathbb{G}, y\mapsto (g(y))_1$
  is an isometry since $\mathbb{G}=\mathbb{G}^{\mathbb{N}}/\ker \|\cdot\|_1$ and it extends $f$. This
  would mean that $\mathbb{G}$ is a separable Banach space of universal disposition for finite-dimensional
  Banach spaces, in contradiction to \cite[Proposition~5.1]{GK11:Gurarii}.
\end{proof}

We conclude the paper with the construction of a sequence of semi-norms on $\mathbb{G}^{\mathbb{N}}$ under
which it is of almost universal disposition for Fr\'{e}chet spaces with a fixed but not necessarily 
increasing sequence of semi-norms.

For this we can use the semi-norms coming from the coordinates, namely, for each $n\in\mathbb{N}$
define $\|x\|_{n}' = \|x(n)\|_{\mathbb{G}}$, $x \in \mathbb{G}^{\mathbb{N}}$, where $\|\cdot\|_{\mathbb{G}}$
is the norm of the Gurari\u{\i} space. We will not prove that this space is unique and universal as
these proofs follow the lines of the corresponding ones for graded Fr\'{e}chet spaces.

In order to shorten the notation, we denote by $X_i:= X/\ker\|\cdot\|_i$ equipped with the norm
$\|\cdot\|_i$ and by $Y_i$ the corresponding quotient of $Y$.

In order to show that $\mathbb{G}^{\mathbb{N}}$ with these semi-norms is of almost universal disposition, 
we need the following

\begin{lemma}\label{lem:eps_embed}
  Let $f_0\colon X \to \mathbb{G}^\mathbb{N}$ be an ($\varepsilon$-)isometric embedding. For all
  $i\in\mathbb{N}$ the mapping
  \begin{displaymath}
    f_0^i\colon X_i \rightarrow \mathbb{G}, \bar{x}\mapsto (f_0(x))_i
  \end{displaymath}
  is an ($\varepsilon$-)isometric embedding and the diagram
  \begin{center}
    \begin{mycd}
      \prod (X/\ker\|\cdot\|_i) \arrow{r}{\prod f_0^i} & \mathbb{G}^\mathbb{N} \\
      X \arrow[hookrightarrow]{u}\arrow[swap]{ru}{f_0}&
    \end{mycd}
  \end{center}
  is commutative.
\end{lemma}

\begin{proof}
  First we show that
  \begin{displaymath}
    f_0^i\colon X_i \rightarrow \mathbb{G}, \bar{x}\mapsto (f_0(x))_i
  \end{displaymath}
  is well-defined. Let $x_1,x_2\in X$ such that $\|x_1-x_2\|_{i}=0$. Then
  \[
    \|(f_0(x_1))_i-(f_0(x_2))_i\| = \|(f_0(x_1-x_2))_i\| \leq (1+\varepsilon) \|x_1-x_2\|_i = 0
  \]
  and hence $(f_0(x_1))_i = (f_0(x_2))_i$, i.e. the mapping is well-defined. It is by definition an
  ($\varepsilon$-)isometric embedding. Finally note that the identity $(f_0(x))_i=f_0^i(\bar{x}^{\|\cdot\|_i})$
  for $i\in\mathbb{N}$ follows analogously.
\end{proof}

\begin{prop}
  Let $(X,\{\|\cdot\|_i\}_{i\in\mathbb{N}})$ and $(Y,\{\|\cdot\|_i\}_{i\in\mathbb{N}})$ be finite-dimensional Fr\'{e}chet spaces with fixed semi-norms, $\varepsilon>0$, $\iota\colon X\to Y$ and
  $f_0\colon X\to\mathbb{G}^\mathbb{N}$ isometric embeddings. Then there is an
  $\varepsilon$-isometric embedding $f\colon Y\to \mathbb{G}^\mathbb{N}$ such that the diagram
  \begin{center}
    \begin{mycd}
      X \arrow{r}{f_0} \arrow[swap]{d}{\iota} & \mathbb{G}^\mathbb{N} \\
      Y  \arrow[swap]{ru}{f} &
    \end{mycd}
  \end{center}
  is commutative.
\end{prop}

\begin{proof}
  As $X$ and $Y$ are finite-dimensional, the quotient spaces $X_i$ and $Y_i$ are finite-dimensional Banach 
  spaces for all $i\in\mathbb{N}$.
  By Lemma~\ref{lem:embed} the mappings $\iota_i\colon X_i \to Y_i$ are isometric embeddings. The same is 
  true for the mappings $f_0^i\colon  X_i\to\mathbb{G}$ by Lemma~\ref{lem:eps_embed}. As the Gurari\u{\i} 
  space is of almost universal disposition for finite-dimensional Banach spaces, there is an
  $\varepsilon$-isometric embedding $f_i\colon Y_i\to\mathbb{G}$ making the diagram
  \begin{center}
    \begin{mycd}
      X_i \arrow{r}{f_0^i} \arrow[swap]{d}{\iota_i} & \mathbb{G}\\
      Y_i \arrow[swap]{ru}{f_i}&
    \end{mycd}
  \end{center}
  commutative. For $y\in Y$, we set $f(y) = \{f_i(\overline{y}^{\|\cdot\|_i})\}_{i\in\mathbb{N}}$, i.e. $f$ is 
  defined so that
  \begin{center}
    \begin{mycd}
      \prod Y_i \arrow{r}{\prod f_i} & \mathbb{G}^\mathbb{N} \\
      Y \arrow[hookrightarrow]{u} \arrow[swap]{ru}{f} &
    \end{mycd}
  \end{center}
  is commutative. Since $f_i$ is an $\varepsilon$-isometric embedding for all $i\in\mathbb{N}$, we get
  \begin{displaymath}
    \|f(y)\|_i' = \|(f(y))_i\| = \|f_i(\overline{y}^{\|\cdot\|_i})\|
    \leq (1+\varepsilon) \|\bar{y}\|_i = (1+\varepsilon) \|y\|_i
  \end{displaymath}
  and by an analogous computation $\|f(y)\|_i' \geq  (1+\varepsilon)^{-1} \|y\|_i$, i.e. $f$ is an 
  $\varepsilon$-isometric embedding. Now let $x\in X$, we have
  \begin{displaymath}
    (f(\iota(x)))_i = f_i\Big(\overline{\iota(x)}^{\|\cdot\|_i}\Big) = f_i(\iota_i(\bar{x}))
    = f_0^i(\bar{x}) = (f_0(x))_i.
  \end{displaymath}
  Hence $f(\iota(x))=f_0(x)$, i.e. $f\restriction X = f_0$.
\end{proof}

\begin{rem}
  Note that in both cases $\mathbb{G}^{\mathbb{N}}/\ker\|\cdot\|_n = \mathbb{G}^{n}$. Therefore all
  neighbourhoods of zero contain straight lines. This means in other words that there is no
  continuous norm on the space $\mathbb{G}^{\mathbb{N}}$ equipped with either of the sequences of
  semi-norms.
\end{rem}

\bibliographystyle{abbrv}
\bibliography{literature}

\end{document}